\documentclass[11pt,reqno]{amsart}

\usepackage{amssymb}
\usepackage{amsmath}
\usepackage{amsfonts}
\usepackage[english]{babel}
\usepackage[T1]{fontenc}
\usepackage[latin1]{inputenc}
\usepackage{amsthm}
\usepackage[all]{xy}

\theoremstyle{plain}
\newtheorem{theorem}{Theorem}
\newtheorem{proposition}[theorem]{Proposition}
\newtheorem{corollary}[theorem]{Corollary}

\theoremstyle{definition}

\theoremstyle{remark}

\newtheorem{remark}[theorem]{Remark}

\newcommand{\Bloch}{{\mathcal B}}
\newcommand{\Blocha}{{{\mathcal B}_\alpha}}
\newcommand{\Blochb}{{{\mathcal B}_\beta}}
\newcommand{\Blochaz}{{{\mathcal B}_{\alpha,0}}}
\newcommand{\Blochbz}{{{\mathcal B}_{\beta,0}}}
\newcommand{\e}{{\mathrm e}}

\newcommand{\UnitDisk}{\mathbb{D}}
\newcommand{\ComplexPlane}{\mathbb{C}}

\newcommand{\wco}{{W_{\psi,\varphi}}}
\newcommand{\wcot}{{\widetilde{W}_{\psi,\varphi}}}

\title[Weighted composition operators]{Weighted composition operators between weak spaces of vector-valued analytic functions}
\author{Mostafa Hassanlou}
\address{M. Hassanlou, Department of Pure Mathematics, Faculty of Mathematical Sciences, University of Tabriz, Tabriz, Iran}
\email{m\_hasanloo@tabrizu.ac.ir}
\author{Jussi Laitila}
\address{J. Laitila, Department of Biosciences, P.O.\ Box 65, FI-00014 University of Helsinki, Helsinki, Finland}
\email{jussi.laitila@helsinki.fi}
\author{Hamid Vaezi}
\address{H. Vaezi, Department of Pure Mathematics, Faculty of Mathematical Sciences, University of Tabriz, Tabriz, Iran}
\email{hvaezi@tabrizu.ac.ir}

\numberwithin{equation}{section}

\begin{document}

\begin{abstract}
We consider weighted composition operators $W_{\psi,\varphi}\colon f\mapsto \psi(f\circ\varphi)$ on spaces of analytic functions on the unit disc, which take values in some complex Banach space.
We provide necessary and sufficient conditions for the boundedness and (weak) compactness of $W_{\psi,\varphi}$ on general function spaces, and in particular on weak vector-valued spaces. As an application, we characterize the weak
compactness of $W_{\psi,\varphi}$ between two different vector-valued Bloch-type spaces. This result appears to be new also in the scalar-valued case.
\end{abstract}

\maketitle

\section{Introduction}

Let $\UnitDisk$ be the open unit disc in the complex plane $\ComplexPlane$ and, for any complex Banach space $X$, let $H(\UnitDisk,X)$ denote the space of all analytic functions
$f\colon \UnitDisk\to X$. For two complex Banach spaces $X$ and $Y$, let $L(X,Y)$ be the space of bounded linear operators from $X$ to $Y$. Then a function $\psi\in H(\UnitDisk, L(X,Y))$ and an analytic self-map $\varphi$ of $\UnitDisk$ induce a linear weighted composition operator $\wco$, which is defined from $H(\UnitDisk, X)$ to $H(\UnitDisk,Y)$, by
\begin{align}\label{eq:wcodef}
(\wco f)(z)=\psi(z)f(\varphi(z)), \quad z\in\UnitDisk.
\end{align}
We  have $\wco=M_\psi C_\varphi$, where $M_\psi$ is the operator-valued multiplier $f\mapsto \psi f$ and $C_\varphi$ is the composition operator $f\mapsto f\circ\varphi$. Weighted compositions appear naturally:
for a large class of Banach spaces $X$,
all linear onto isometries between $X$-valued $H^\infty$ spaces are of the form \eqref{eq:wcodef} for suitable $\psi$ and $\varphi$;
see \cite{CJ90} and \cite{Li90}.
Properties of multipliers $M_\psi$ and composition operators $C_\varphi$ have been widely studied on various spaces of vector-valued analytic functions, including vector-valued Hardy, Bergman, BMOA and Bloch spaces; see, for example, 
\cite{BDL01, L05, LT06, LST98, M03} and the recent survey \cite{LT14}. This study has further been extended to weighted composition operators; see \cite{BGJW12, HVW15, LT09, M08, M11, W11}.

The study of composition operators on vector-valued function spaces involves some basic results which hold for large classes of function spaces. For example, if $X$ is infinite dimensional, then composition operators are usually never compact on an $X$-valued analytic function space. On the other hand, for the weak compactness of $C_\varphi$ it is necessary that $X$ is reflexive and that $C_\varphi$ is weakly compact on the corresponding scalar-valued function space. For the latter statement a partial converse holds for the so-called weak vector-valued function spaces: for any Banach space $E$ of analytic functions $f\in H(\UnitDisk, \ComplexPlane)$ such that $E$ contains the constant functions and the closed unit ball $B_E$ is compact in the compact open topology 
of $\UnitDisk$, the weak space $wE(X)$ (modelled on $E$) is the space of functions $f\in H(\UnitDisk,X)$ such that
\begin{align}\label{eq:defweakspace}
\Vert f \Vert_{wE(X)}=\sup_{\Vert x^* \Vert_{X^*} \le 1}\Vert x^* \circ f \Vert_{E} < \infty.
\end{align}
See \cite{BDL01, LT14} for the above results.

The first aim of this paper is to extend such general necessary and sufficient conditions to weighted composition operators between two different function spaces, where the weight function $\psi$ is scalar valued, i.e., $\psi\in H(\UnitDisk, \ComplexPlane)$. 
Consequently, we obtain characterizations of the weak compactness of weighted composition operators between many weak vector-valued spaces, such as weak Hardy and Bergman spaces. 

In the second part, we apply the above results in characterizing the weak compactness of weighted composition operators between two different Bloch-type function spaces. Here the main part consists of a scalar-valued result, which also appears to be new.

Recall that an operator $T\in L(X,Y)$ is (weakly) compact if for any bounded sequence $(x_n)$ in $X$ there is a (weakly) convergent subsequence $(T x_{n_k})$ in $Y$. Any compact operator is clearly weakly compact.

\section{Weighted composition operators on general vector-valued function spaces}

In this section we consider weighted composition operators between general spaces of vector-valued analytic functions.
Most arguments are fairly straightforward extensions of known results for composition operators from earlier work (such as \cite{LST98}, \cite{BDL01} and \cite{LT14}) but we present the details for completeness. Below we assume that the weight function $\psi$ is complex valued ($\psi\in H(\UnitDisk, \ComplexPlane)$) and nonzero ($\psi(z)\ne 0$ for some $z\in\UnitDisk$).

Suppose that $A$ is a Banach space of analytic functions $\UnitDisk \to \ComplexPlane$
and let $A(X)$ be an associated vector-valued  Banach space of analytic functions $\UnitDisk \to X$,
where $X$ is a complex Banach space.
Assume that the following properties hold for the pair $(A,A(X))$ for all Banach spaces $X$ (see \cite[Section 2]{LT14}):

\begin{itemize}
\item[(a1)] The constant maps $f(z) \equiv c$ belong to $A$ for all $c \in \mathbb C$.

\item[(a2)]  $f \mapsto f \otimes x$ defines a bounded linear operator
$J_x\colon A \to A(X)$
for  any $x \in X$, where $(f \otimes x)(z) = f(z)x$ for $z \in \mathbb{D}$.

\item[(a3)]  $f \mapsto x^* \circ f$ defines a bounded linear operator $Q_{x^*}\colon A(X) \to A$ for any $x^* \in X^*$.

\item[(a4)] The point evaluations $\delta_z$,  where $\delta_z(f) = f(z)$
for $f \in A(X)$, are bounded $A(X) \to X$ for all  $z \in \mathbb D$.
\end{itemize}

It follows from these assumptions that the vector-valued constant maps $z\mapsto f_x(z)=1\otimes x\equiv x$ belong to $A(X)$ for all $x\in X$; see \cite{LT14}.

Suppose that $A(X)$ and $B(X)$ are two spaces of analytic functions $\UnitDisk\to X$ such that the pairs $(A,A(X))$ and $(B,B(X))$ satisfy
(a1) -- (a4). The following result extends Corollary 2 of \cite{LT14} and Proposition 1 of \cite{BDL01} to weighted composition operators. The result involves general operator ideals in the sense of Pietsch \cite{P80}, although the remainder of the paper mainly concerns the closed ideals of compact and weakly compact operators (but see Remark \ref{rmk1} below).
We will for clarity use the notation
$\wcot$ for the weighted composition operator between vector-valued spaces $A(X) \to B(X)$ and $\wco$ for the operator between corresponding complex-valued spaces $A\to B$.
We denote by $I_X$ the identity operator of $X$.

\begin{proposition}\label{prop1}
Let $X$ be a complex Banach space.
Let $\psi\in H(\UnitDisk, \ComplexPlane)$ be nonzero and let $\varphi$ be an analytic self-map of $\UnitDisk$.
\begin{enumerate}
\item
If $\wcot$ is bounded $A(X) \to B(X)$, then $\wco$ is bounded
$A \to B$.
\item
If
$\wcot\colon A(X) \to B(X)$ belongs to an operator ideal, 
then $I_X$ and
$\wco\colon A \to B$ belong to the same operator ideal.
\end{enumerate}
\end{proposition}

\begin{proof}
Let $x\in X$, $x^*\in X^*$ be norm-1 vectors such that $\langle x^*, x \rangle = 1$, let $z_0\in \UnitDisk$ be a point such that $\psi(z_0)\ne 0$, and  let $j_1$, $j_2$, $j_3$, $j_4$ be the linear operators
\begin{align*}
&j_1\colon A\to A(X), \qquad f\mapsto f\otimes x;\\
&j_2\colon B(X)\to B, \qquad f\mapsto x^*\circ f;\\
&j_3 \colon X\to A(X), \qquad x\mapsto f_x;\\
&j_4\colon B(X)\to X, \qquad f\mapsto f(z_0)/\psi(z_0).\\
\end{align*}
By (a1) -- (a4), all these operators are bounded, and
\begin{align*}
\wco&=j_2 \circ\wcot \circ j_1;\\
I_X&=j_4\circ \wcot \circ j_3.
\end{align*}
.
\end{proof}

The following result settles the question of compactness of $\wcot$ between vector-valued spaces satisfying (a1) -- (a4).

\begin{proposition}\label{prop2}
Let $\psi\in H(\UnitDisk, \ComplexPlane)$ be nonzero and let $\varphi$ be an analytic self-map of $\UnitDisk$.
Then $\wcot$ is compact $A(X) \to B(X)$ if and only if $X$ is finite dimensional and $\wco$ is compact
$A \to B$.
\end{proposition}

\begin{proof}
Proposition \ref{prop1} (2) implies that if
$\wcot$ is compact, then $X$ is finite dimensional and $\wco$ is compact.
The proof of the converse direction is a straightforward modification of the proof of \cite[Prop.~7]{LT14} and is thus omitted.
\end{proof}

Propositions \ref{prop1} and \ref{prop2} apply to a large number of function spaces, such as $X$-valued Hardy, Bergman, BMOA, and Bloch spaces as well as to their weak versions (see, for example, \cite{LT14} for the definitions of these function spaces). 
Before explicitly discussing any corollaries, let us consider the general class of weak vector-valued function spaces \eqref{eq:defweakspace}, introduced by Bonet, Doma\'nski and Lindstr\"om \cite{BDL01} in the context of composition operators.
Suppose that  $E$ is a Banach space of analytic functions $f\in H(\UnitDisk, \ComplexPlane)$ satisfying
the following conditions:

\begin{itemize}
\item[(b1)] $E$ contains the constant functions,

\item[(b2)] the closed unit ball $B_E$ is com\-pact in the compact open topology $\tau_{co}$
of $\UnitDisk$.
\end{itemize}
Then the weak space defined in \eqref{eq:defweakspace} is a Banach space; see \cite{BDL01}.
Moreover, by \cite[Lemma 13]{LT14}, properties (b1) and (b2) imply the assumptions (a1) -- (a4), so that  Propositions \ref{prop1} and \ref{prop2} apply to weak spaces.

The next result provides converses to Proposition \ref{prop2} (1) and (2) for bounded and weakly compact weighted composition operators between two different weak spaces. It also extends \cite[Proposition 11]{BDL01} to weighted composition operators.
Recall here that the identity operator $I_X$ of $X$ is weakly compact if and only if $X$ is reflexive.

\begin{proposition}\label{prop3}
Let $\psi\in H(\UnitDisk, \ComplexPlane)$ and let $\varphi$ be an analytic self-map of $\UnitDisk$.
Suppose that $E_1$ and $E_2$ are spaces of analytic functions on $\UnitDisk$ that satisfy assumptions (b1) and (b2).
\begin{enumerate}
\item
If $\wco\colon E_1 \to E_2$ is bounded, then $\wcot\colon wE_1(X) \to wE_2(X)$ is bounded.
\item
If $X$ is a reflexive Banach space and $\wco\colon E_1\to E_2$ is com\-pact,
then  $\wcot$ is weakly com\-pact $wE_1(X) \to wE_2(X)$.
\end{enumerate}
\end{proposition}

For the proof of Proposition \ref{prop3} we need the following facts from \cite{BDL01}:
The assumptions (b1) and (b2) imply that the point evaluation maps $\delta_z$ belong to $E^*$ for each $z\in \UnitDisk$ and
that $E$ is the dual space of $V_E$, where $V_E$ is the closed linear span in $E^*$ of $\{\delta_z\in E^*\colon z\in\UnitDisk\}$. Moreover, by \cite[Lemma 10]{BDL01}, there is an isometric isomorphism $\chi \colon L(V_E,X)\to wE(X)$, so that
\begin{equation}\label{eq:linearization}
(\chi (T))(z) = T(\delta_z), \quad  (\chi^{-1}(f))(\delta_z) = f(z),
\end{equation}
hold for $T \in L(V_E,X)$, $f\in wE(X)$ and $z\in\UnitDisk$.

\begin{proof}[Proof of Proposition \ref{prop3}]
For (1), note that
\begin{align*}
\Vert x^* \circ (\wcot f) \Vert_{E_2} = \Vert \wco (x^* \circ f)\Vert_{E_2} \le \Vert \wco \Vert \cdot \Vert x^* \circ f \Vert_{E_1},
\end{align*}
so that $\Vert \wcot \Vert \le \Vert \wco \Vert$.

We next prove (2).
Assume that $\wco\colon E_1 \to E_2$ is com\-pact. Its adjoint
$(\wco)^*\colon E_2^* \to E_1^*$ satisfies
\[
(\wco)^*(\delta_z)=\psi(z)\delta_{\varphi(z)}, \quad z\in\UnitDisk,
\]
so that  $(\wco)^*(V_{E_2})\subset V_{E_1}$. We obtain the factorization
$\wcot = \chi_2\circ S_{\psi,\varphi}\circ  \chi_1^{-1}$, where $S_{\psi,\varphi}$ is the
operator composition map
\[
T\mapsto I_{L(V_{E_1},X)}\circ T\circ (\wco)^*|_{V_{E_2}}; \quad L(V_{E_2}, X)\to L(V_{E_1}, X),
\]
and $\chi_1\colon L(V_{E_1},X) \to wE_1(X)$  and $\chi_2\colon L(V_{E_2},X) \to wE_2(X)$ are isometric isomorphisms from 
\eqref{eq:linearization}.
Indeed,
\begin{align*}
((\chi_2\circ S_{\psi,\varphi}\circ  \chi_1^{-1})(f))(z)&=
(\chi_2(\chi_1^{-1}(f)\circ  (\wco)^*|_{V_{E_2}})(z)\\
&=\chi_1^{-1}(f)((\wco)^*(\delta_z))\\
&=\chi_1^{-1}(f)(\psi(z)\delta_{\varphi(z)})\\
&=\psi(z)\chi_1^{-1}(f)(\delta_{\varphi(z)})\\
&=\psi(z)(f(\varphi(z)),
\end{align*}
for $f\in E_1$ and $z\in\UnitDisk$.

Since $(\wco)^*|_{V_{E_2}}$ is a com\-pact operator $V_{E_2} \to V_{E_1}$ by duality,
and $I_X$ is weakly com\-pact by reflexivity of $X$, it follows from \cite[Theorem 2.9]{ST92} (and \cite[Remark 2.4]{ST06}), that the operator composition  $S_{\psi,\varphi}$ is weakly com\-pact $L(V_{E_2},X) \to L(V_{E_1},X)$.
Consequently $\wcot$ is weakly com\-pact $wE_1(X) \to wE_2(X)$.
\end{proof}

As concrete examples, let us next consider weak versions of two classical function spaces: the Hardy space $H^1$ and the Bergman space $A^1$ 
(see, e.g., \cite{Z90} for the definitions of $H^1$ and $A^1$).
Both of these spaces satisfy (a2) and (b2), so the associated weak spaces $wH^1(X)$ and $wA^1(X)$ are Banach spaces for all $X$. However, they are quite different from the usual ''strong'' spaces $H^1(X)$ and $A^1(X)$; see \cite{L05, LT06, LTW09}.

Compactness and weak compactness of $\wco$ are well understood on $H^1$ and $A^1$: Function-theoretic characterizations exist for the compactness of $\wco$ on both $H^1$ and $A^1$ \cite{CH01, CZ04}. In \cite{CH01} it was further shown that  every weakly compact $\wco$ on $H^1$ is compact.
Moreover, because $A^1$ is isomorphic to $\ell^1$, it has the Schur property, and so every weakly compact linear operator on $A^1$ is compact; see \cite[p.~302]{LST98}. Using Proposition \ref{prop3}, we get the following extension of these scalar results.

\begin{corollary}\label{cor:examples} Let $\psi\in H(\UnitDisk, \ComplexPlane)$ be nonzero and let $\varphi$ be an analytic self-map of $\UnitDisk$. Then
\begin{enumerate}
\item $\wcot$ is weakly compact on $wH^1(X)$ if and only if $X$ is reflexive and
$\wco$ is compact on $H^1$.
\item $\wcot$ is weakly compact on $wA^1(X)$ if and only if $X$ is reflexive and
$\wco$ is compact on $A^1$.
\end{enumerate}
\end{corollary}

Corollary \ref{cor:examples} could include more examples. For example, (weak) compactness of $\wco$ between weighted spaces $H^\infty_v$ is well known \cite{CD99, CH00}. Also the weak space $w H^\infty_v(X)$ is well defined (and coincides with the usual ''strong'' space $H^\infty_v(X)$). However, weak compactness of the more general ''operator-weighted'' compositions $\wcot$ (where $\psi\in H(\UnitDisk, L(X,Y))$) between $H^\infty_v(X)$ spaces has been characterized in \cite{LT09}.

It is worthwhile to note that the general framework (Propositions \ref{prop1}, \ref{prop2} and \ref{prop3}) does not in general hold for operator-weighted composition operators. For example, an operator-weighted composition between weighted $H^\infty(X)$ spaces can be compact even if $X$  is infinite dimensional \cite{LT09}. Indeed, the present knowledge of operator-valued composition opeators is quite rudimentary and, for example, their boundedness and compactness on Hardy spaces $H^p(X)$ is still an open question; see the discussion in \cite[Section 7]{LT14}.

\begin{remark}\label{rmk1}
An analogue of Proposition \ref{prop3} (2) holds for the ideal of weakly conditionally compact operators.
Recall that an operator $T\in L(X,Y)$ is weakly conditionally compact if for any bounded sequence $(x_n)$ in $X$ there is a weakly Cauchy subsequence $(T x_{n_k})$ in $Y$.
Then, if $X$ does not contain an isomorphic copy of $\ell^1$ and $\wco\colon E_1\to E_2$ is com\-pact,
it follows that $\wcot$ is weakly conditionally com\-pact $wE_1(X) \to wE_2(X)$. The proof is analogous to Proposition \ref{prop3} (2), but uses \cite{LS99} instead to deduce that the operator composition $S_{\psi,\varphi}$ is weakly conditionally compact.
\end{remark}


\section{Weighted composition operators between Bloch-type spaces}

In this section we apply the results from the previous section to weighted composition operators between Bloch-type spaces of analytic functions.
For $\alpha\in (0,\infty)$ and a complex Banach space $X$, the vector-valued Bloch-type space $\Blocha(X)$ is the Banach space of functions $f\in H(\UnitDisk, X)$ such that
\begin{align}\label{eq:defblohtypespace}
\Vert f \Vert_{\Blocha(X)}=\Vert f(0) \Vert_X + \sup_{z \in\UnitDisk}\Vert f'(z) \Vert_{X}(1-|z|^2)^\alpha < \infty.
\end{align}
In the special case $X=\ComplexPlane$, we get the usual scalar-valued Bloch-type spaces $\Blocha=\Blocha(X)$. In the case $\alpha=1$, we get the Bloch space
$\Bloch(X)=\Bloch_1(X)$ (see, for example, \cite{AB03}).

Boundedness and compactness of weighted composition operators between scalar-valued Bloch-type spaces $\Blocha$ is well known. Ohno, Stroethoff and Zhao \cite{OSZ03} provided the following function-theoretic characterizations, which
involve the quantities
\begin{align*}
q_1(\alpha,\beta,z)=\frac{(1-|z|^2)^{\beta}}{(1-|\varphi(z)|^2)^{\alpha}} |\psi(z)||\varphi'(z)|;\\
q_2(\beta,z)=(1-|z|^2)^{\beta} \log \frac{1}{1-|\varphi(z)|^2} |\psi'(z)|;\\
q_3(\alpha,\beta,z)=\frac{(1-|z|^2)^{\beta}}{(1-|\varphi(z)|^2)^{\alpha -1}} |\psi'(z)|,
\end{align*}
for $\alpha,\beta\in (0,\infty)$:
\begin{enumerate}
\item If $0< \alpha < 1$, then $\wco\colon \Blocha \to \Blochb$ is bounded (respectively compact)  if and only if
\begin{align*}\sup_{z\in \UnitDisk} q_1(\alpha,\beta,z)<\infty\end{align*}
(respectively
\begin{align*}
\lim_{|\varphi(z)|\to 1} q_1(\alpha,\beta,z)=0).
\end{align*}
\item If $\alpha = 1$, then $\wco\colon \Bloch \to \Blochb$ is bounded (respectively compact) if and only if
\begin{align*}
\sup_{z\in \UnitDisk} q_1(1,\beta,z)<\infty\quad \textrm{and}\quad \sup_{z\in\UnitDisk}q_2(\beta,z)<\infty
\end{align*}
(respectively
\begin{align*}
\lim_{|\varphi(z)|\to 1} q_1(1,\beta,z)=0\quad \textrm{and}\quad \lim_{|\varphi(z)|\to 1}q_2(\beta,z)=0).
\end{align*}
\item If $1< \alpha < \infty$, then $\wco\colon \Blocha \to \Blochb$ is bounded  (respectively compact) if and only if
\begin{align*}
\sup_{z\in \UnitDisk} q_1(\alpha,\beta,z)<\infty\quad \textrm{and}\quad \sup_{z\in\UnitDisk}q_3(\alpha,\beta,z)<\infty
\end{align*}
(respectively
\begin{align*}
\lim_{|\varphi(z)|\to 1}q_1(\alpha,\beta,z)=0\quad \textrm{and}\quad\lim_{|\varphi(z)|\to 1}q_3(\alpha,\beta,z)=0).
\end{align*}
\end{enumerate}

The following theorem extends the scalar results to vector-valued Bloch-type spaces.
For composition operators, the case $\alpha=\beta=1$ is from \cite[Thm.~4]{LST98} and the case $\alpha=\beta \in (0,\infty)$ follows from \cite[Prop.~11]{BDL01}.
Our proof applies a similar approach to that of \cite{LST98}.

\begin{theorem}\label{thm:blochtype}
Let $\psi\in H(\UnitDisk, \ComplexPlane)$ be nonzero and let $\varphi$ be an analytic self-map of $\UnitDisk$.
Let $\alpha,\beta\in (0,\infty)$. Let $X$ be a complex Banach space.
\begin{enumerate}
\item $\wco$ is weakly compact $\Blocha\to \Blochb$ if and only if it is compact.
\item $\wcot$ is bounded $\Blocha(X)\to \Blochb(X)$ if and only if $\wco$ is
bounded $\Blocha\to \Blochb$.
\item $\wcot$ is weakly compact $\Blocha(X)\to \Blochb(X)$ if and only if $\wco$ is (weakly) compact $\Blocha\to \Blochb$ and $X$ is reflexive.
\end{enumerate}
\end{theorem}

For the proof of Theorem \ref{thm:blochtype}, we will need the following facts.
For $\alpha \in (0,\infty)$, let $\Blochaz$ denote the little Bloch-type space, which is the closed subspace of $\Blocha$ for which
\begin{align*}
\lim_{|z|\to 1}| f'(z) |(1-|z|^2)^\alpha=0.
\end{align*}
First, by \cite[Thm.~15]{Z93}, the Bergman space $A^1$ is the dual space of $\Blochaz$ for all $\alpha$ under the pairing
\begin{align}\label{eq:duality}
\langle f, g \rangle_\alpha = \lim_{r \rightarrow 1} \int_{\mathbb{D}} f(rz) \overline{g(rz)} (1-|z|^2)^{\alpha-1} \,\mathrm  dA(z), \quad f \in
\Blochaz , g \in A^1,
\end{align}
where $\mathrm dA(z)$ is the normalized area measure on $\UnitDisk$. Second, each $\Blochaz$ has the Dunford-Pettis property (DPP). Recall that a Banach space $E$ has the DPP whenever every weakly compact operator from $E$ to some Banach space is completely continuous, i.e., maps weakly null sequences to norm null sequences. The fact that $\Blochaz$ has the DPP follows from the above duality and the fact that $A^1\sim \ell^1$ has the Schur property; see, for example, \cite{Diestel, LST98}.



\begin{proof}[Proof of Theorem \ref{thm:blochtype}]
It is easy to see that each $\Blocha(X)$ is a weak space of the form \eqref{eq:defweakspace}. In fact, $\Blocha(X)=w\Blocha(X)$, with equivalent norms, where $w\Blocha(X)$ is modelled on $\Blocha$. Therefore parts (2) and (3) of Theorem \ref{thm:blochtype} follow from part (1) and Propositions \ref{prop1} and \ref{prop3}. We thus only need to prove part (1).

Let first $\alpha,\beta\in (0,\infty)$ and
suppose that $\wco$ is weakly compact, but not compact, $\Blocha \to \Blochb$. Then there are $\varepsilon > 0$ and a sequence $(z_n)\subset \UnitDisk$ so that $|\varphi(z_n)|\to 1$ and $q_1(\alpha, \beta, z_n)\ge \varepsilon$ for all $n$. For $n\ge 0$, define the functions $f_n\in H(\UnitDisk, \ComplexPlane)$ by
\begin{align*}
f_n(z)=\frac{(1-|\varphi(z_n)|^2)^2}{(1-\overline{\varphi(z_n)}z)^{\alpha+1}}-\frac{1-|\varphi(z_n)|^2}{(1-\overline{\varphi(z_n)}z)^{\alpha}},
\end{align*}
for $z\in\UnitDisk$. In \cite[p.~202]{OSZ03} it has been shown that
the sequence $(f_n)$ is uniformly bounded in $\Blocha$,
$f_n(z)\to 0$ uniformly on compact subsets of $\UnitDisk$,
$f_n(\varphi(z_n))=0$, and $f_n'(\varphi(z_n))=\overline{\varphi(z_n)}/(1-|\varphi(z_n)|^2)^\alpha$. Thus
\begin{align}\label{eq:blochtype1}
\begin{split}
\Vert \wco f_n \Vert_{\Blochb}
&\ge
| (\wco f_n)'(z_n) |(1-|z_n|)^\beta \\
&=|\varphi(z_n)| q_1(\alpha,\beta,z_n) \\
&\ge \varepsilon/2,
\end{split}
\end{align}
for all $n$ large enough.
We further have $f_n\in \Blochaz$ for all $n$. Let $\widehat f_n(j)$ denote the $j$th Taylor coefficient of $f_n$ and let $p_N(z)=\sum_{j=0}^N a_j z^j$ be a polynomial.
Since $f_n(z)\to 0$ uniformly on compact subsets of $\UnitDisk$, we have $\widehat f_n(j) = f_n^{(j)}(0)/j!\to 0$ as $n\to\infty$.
By \eqref{eq:duality},
\begin{align*}
\langle f_n, p_N \rangle_\alpha &= \lim_{r \rightarrow 1} \int_0^{2\pi}\int_0^1\sum_{k=0}^\infty\sum_{j=0}^N \widehat f_n(k)\overline{a_j}r^{k+j}s^{k+j+1}\e^{i(k-j)\theta} (1-s^2)^{\alpha-1} \,\mathrm  ds\frac{\mathrm  d\theta}{2\pi}\\
&=\sum_{j=0}^N \widehat f_n(j)\overline{a_j} \int_0^1 s^{2j+1}(1-s^2)^{\alpha-1} \,\mathrm ds \to 0,
\end{align*}
as $n\to\infty$. Since the polynomials are dense in $A^1$, this means that $f_n\to 0$ weakly in $\Blochaz$. 
Weak compactness of $\wco$ and the DPP of $\Blochaz$ now imply that $\wco f_n\to 0$ in $\Blochbz$ and thus in $\Blochb$. This contradicts \eqref{eq:blochtype1}.

We next divide the proof into 3 cases: (i) $\alpha\in (0,1)$, (ii) $\alpha=1$, and (iii) $\alpha\in (1,\infty)$.
The above argument proves the case (i). Cases (ii) and (iii) are completed similarly, by using functions
\begin{align*}
g_n(z)=\frac{-1}{\log(1-|\varphi(z_n)|^2)}\left(3\left(\log\frac{1}{1-\overline{\varphi(z_n)}z}\right)^2-2\left(\log\frac{1}{1-\overline{\varphi(z_n)}z}\right)^3\right),
\end{align*}
respectively,
\begin{align*}
h_n(z)=(\alpha+1)\frac{1-|\varphi(z_n)|^2}{(1-\overline{\varphi(z_n)}z)^\alpha}-\alpha\frac{(1-|\varphi(z_n)|^2)^2}{(1-\overline{\varphi(z_n)}z)^{\alpha+1}},
\end{align*}
instead of $f_n$.
\end{proof}

\section*{Acknowledgement}
We thank Hans-Olav Tylli for his comments on a preliminary version of the manuscript.


\end{document}